\documentclass[reqno,11pt]{amsart}

\usepackage{amsthm, amsmath, amssymb, stmaryrd}
\usepackage[utf8]{inputenc}
\usepackage[T1]{fontenc}

\usepackage[hyphenbreaks]{breakurl}
\usepackage[hyphens]{url}
\usepackage{systeme}
\usepackage[shortlabels]{enumitem}
\usepackage[hidelinks]{hyperref}
\usepackage{microtype}
\usepackage{algorithm, algcompatible}

\usepackage{bm}
\usepackage[margin=1in]{geometry}

\usepackage[textsize=scriptsize,backgroundcolor=orange!5]{todonotes}

\usepackage[noabbrev,capitalize,sort]{cleveref}
\crefname{equation}{}{}
\numberwithin{equation}{section}

\usepackage{mathtools}
\newtheorem{theorem}{Theorem}[section]

\newtheorem{lemma}[theorem]{Lemma}

\newtheorem{corollary}[theorem]{Corollary}

\theoremstyle{definition}

\theoremstyle{remark}
\newtheorem*{remark}{Remark}

\crefname{algorithm}{Procedure}{Procedures}

\makeatletter
\newcommand\thankssymb[1]{\textsuperscript{\@fnsymbol{#1}}}
\makeatother

\author[Hung-Hsun Hans Yu]{Hung-Hsun Hans Yu}

\address{Department of Mathematics, Princeton University, Princeton, NJ 08544}
\email{hansonyu@princeton.edu}

\title{Deciding Foot-sortability and Minimal 2-bounded Non-foot-sortable Sock Orderings}

\begin{document}

\maketitle

\begin{abstract}
	A sock ordering is a sequence of socks with different colors.
 A sock ordering is foot-sortable if the sequence of socks can be sorted by a stack so that socks with the same color form a contiguous block.
 The problem of deciding whether a given sock ordering is foot-sortable was first considered by Defant and Kravitz, who resolved the case for alignment-free 2-uniform sock orderings.
 In this paper, we resolve the problem in a more general setting, where each color appears in the sock ordering at most twice.
 A key component of the argument is a fast algorithm that determines the foot-sortability of a sock ordering of length $N$ in time $O(N\log N)$, which is also an interesting result on its own.
\end{abstract}

\section{Introduction}
A \emph{stack} is a data structure where, at each step, one can do one of the following three options: put a new element to the top, examine the element at the top, and remove the top element.
The stack-sorting problem is a classical problem that asks to sort a sequence of integers using a stack.
This was, for example, considered in Knuth's book \textit{The Art of Computer Programming} \cite{Knuth}.
This problem has been well-understood, with a linear-time algorithm determining the stack-sortability of a sequence, and also a simple criterion for stack-sortability.
The criterion is that a sequence is stack-sortable if and only if there is not a subsequence $b,c,a$ with $a<b<c$ \cite{Knuth}.
This property is often called \emph{231-avoiding}, though we will not use the terminology in this paper.
The linear-time algorithm and the characterization of stack-sortable sequences both give us a way to count the number of stack-sortable permutations of $[n]$, and strikingly, the number is exactly the $n$-th Catalan number\cite{Knuth}.

Recently, inspired by a real-life application involving socks, Defant and Kravitz \cite{DK22} consider the stack-sorting problem for set partitions.
In their terminology, a \emph{sock ordering} is a sequence of socks with different colors, and a sock ordering is \emph{foot-sortable} if one can sort it with a stack so that socks with the same color are contiguous.
We refer the readers to their paper for a fantastic real-life demonstration of how foot-sorting works.
In this paper, we will represent a sock ordering with a string $s_1s_2\cdots s_n$, where we represent different colors by different letters.
We will thus interchangeably call $s_1\cdots s_n$ a string or a sock ordering, and any $s_i$ a color or a letter.

Defant and Kravitz ask if one could come up with a simple description of foot-sortable sock orderings analogous to the description for stack-sortable sequences.
There are two equivalent ways of stating this formally, both of which will be used in this paper.
The first is a pattern-avoiding description of foot-sortable sock orderings.
We say that a sock ordering $s_1\cdots s_n$ contains another sock ordering $s_1'\cdots s_m'$ as a \emph{pattern} if there is a subsequence $s_{i_1}\cdots s_{i_m}$ of $s_1\cdots s_n$ so that there is a bijection $f:\{s_{i_1},\ldots, s_{i_m}\}\to \{s_1',\ldots, s_m'\}$ with $f(s_{i_j})=s_j'$ for all $j\in [m]$.
Then we wish to find an explicit list (as short as possible) so that a sock ordering is not foot-sortable if and only if it contains some sock ordering in the list as a pattern.
This pattern-avoiding formulation is used more commonly in the literature, but the next formulation is also useful.
A non-foot-sortable sock ordering is \emph{minimal} if deleting any sock makes the sock ordering foot-sortable.
Then it is clear that the list of minimal non-foot-sortable sock orderings can be used as the pattern-avoiding description for the foot-sortable sock orderings, and thus it is equivalent to find all minimal non-foot-sortable sock orderings.

Given that the original motivation involves socks, it is natural to restrict to the sock orderings where each color appears exactly twice.
Defant and Kravitz call such sock orderings \emph{2-uniform}, and they give fourteen 2-uniform patterns that characterize the 2-uniform foot-sortable sock orderings via pattern-avoiding as long as the sock ordering is ``alignment-free'' \cite{DK22}.
Here, a sock ordering is \emph{alignment-free} if it does not contain the pattern $aabb$.
One of the main results of this paper is to lift the assumption that the sock ordering is alignment-free while relaxing the 2-uniform condition by a little bit.
A sock ordering is said to be \emph{2-bounded} if each color appears at most twice.
In this paper, we will give a full classification of minimal non-foot-sortable 2-bounded sock orderings.
The precise list of such sock orderings will be given in \cref{lemma:sporadic} and \cref{lemma:infinite}.

\begin{theorem}\label{theorem:all-minimal}
    There are $14$ sporadic minimal non-foot-sortable 2-bounded sock orderings, and $4$ infinite families of such sock orderings.
\end{theorem}

The statement might come as striking already, as Defant and Kravitz \cite{DK22} prove that there are only finitely many of them to avoid for alignment-free 2-bounded sock orderings.
Defant and Kravitz ask in \cite[Question 5.1]{DK22} whether there are only finitely many patterns to avoid for general sock orderings or for general 2-bounded sock orderings, and \cref{theorem:all-minimal} answers this negatively.

Just as the linear-time algorithm for deciding the stack sortability gives a characterization of stack-sortable sequences, \cref{theorem:all-minimal} relies on a fast algorithm that determines the foot-sortability of a sock ordering.
In this case, we do not need to assume 2-boundedness.

\begin{theorem}\label{theorem:algorithm}
    There exists an algorithm running in $O(N\log N)$ time which determines whether a given sock ordering of length $N$ is foot-sortable or not.
\end{theorem}

It is likely that using this algorithm, we may be able to classify all the minimal foot-sortable sock orderings using the argument for \cref{theorem:all-minimal}.
As some preliminary attempt shows that there might be many more sporadic examples and infinite families, we do not pursue in that direction and only focus on the 2-bounded cases.

This paper is structured as follows.
In \cref{section:prelim}, we briefly go through some results that will be needed later on.
In \cref{section:algorithm}, we first prove some useful results that cut down the search space dramatically, and then show how to put them together to get an algorithm that runs in $O(N\log N)$ time using some standard data structures and techniques.
Lastly, in \cref{section:min-2-bdd}, we use the algorithm constructed in \cref{section:algorithm} to find all the minimal 2-bounded non-foot-sortable sock orderings.

\section{Preliminaries}\label{section:prelim}
As noted in \cite{DK22}, foot-sorting for sock orderings can be viewed as stack-sorting for set partitions.
We quickly recall a classical result in stack-sorting mentioned in the introduction.

\begin{theorem}[Knuth \cite{Knuth}]
A sequence of numbers can be sorted in increasing order using a stack if and only if there are no three numbers $a<b<c$ that appear in the sequence in the order of $b,c,a$.
\end{theorem}

The theorem gives a criterion for stack-sortability that is easy to test.
Thus, in the foot-sorting setting, if we furthermore require that the colors appear in a particular order, then we may tell if it is possible using a single foot with this criterion.
This has already been noted in the original paper \cite{DK22}, but we will rephrase it in a different language.

Let $S$ be a string representing a sock ordering, $\Gamma$ be the alphabet, and $\leq$ be a partial order on $\Gamma$.
The pair $(S,\leq)$ is \emph{foot-sortable} if it is possible to sort the sock sequence using one foot, so that if $a<b$ then the socks of color $a$ end up before the socks of color $b$.
Then \cite[Theorem 1.1]{DK22} may be rephrased as follows.

\begin{theorem}[Defant--Kravitz \cite{DK22}]\label{theorem:criterion}
    Suppose $S$ is a string that represents a sock ordering. 
    Then the sock ordering is foot-sortable if and only if there exists a total order $\leq_{\ell}$ on the alphabet $\Gamma$ so that no three letters $a<_{\ell}b<_{\ell}c\in \Gamma$ appear in $S$ in the order of $b,c,a$.
\end{theorem}

Although the foot-sortability of a pair $(S,\leq_{\ell})$ can be tested easily, it is unclear how to find such $\leq_{\ell}$ efficiently if it exists.
The algorithm will focus on cutting down the search space for such total orders.

\section{$O(N\log N)$ algorithm for deciding foot-sortability}\label{section:algorithm}
In this section, we provide an algorithm that decides the foot-sortability of any given sock ordering in $O(N\log N)$ time.
We will first discuss the strategy of the algorithm along with some lemmas that will help the analysis of the algorithm.
The description of the algorithm with a full analysis of it will be provided later in \cref{subsection:algorithm}.

Let $S$ be a string representing a sock ordering, and denote by $\leq$ a partial order on the alphabet $\Gamma$.
To decide foot-sortability efficiently, our main task is to find, if the sock ordering is foot-sortable, a total order $\leq_{\ell}$ efficiently such that $(S,\leq_{\ell})$ is foot-sortable.
The total order $\leq_{\ell}$ will be built step by step: at each step, we will find a specific letter that we may assume to be the next minimum without loss of generality.
For our convenience, for any subset $\Gamma'\subseteq \Gamma$ and any partial order $\leq$ on $\Gamma$, we say that $(S,\leq)$ is \emph{$\Gamma'$-decided} if either $(S,\leq)$ is not foot-sortable, or there is a total order $\leq_{\ell}$ extending $\leq$ with its minimum element in $\Gamma'$ such that $(S,\leq_{\ell})$ is foot-sortable.
With a slight abuse of notation, if $\Gamma'$ is a singleton $\{a\}$, then we will also say that $(S,\leq)$ is \emph{$a$-decided}.
With this definition, our task is now to find a specific letter $a$ so that $(S,\leq)$ is $a$-decided.

To see why having such a letter helps us, we will reduce the size of the problem with the help of $a$.
Given a pair $(S,\leq)$ where $S$ is a string representing a sock ordering and $\leq$ is a partial order on the alphabet, an \emph{$a$-reduction} is the following procedure:
\begin{enumerate}
    \item Test if $a$ is a minimal element with respect to $\leq$.
    If not, then the reduction fails.
    \item Extend the partial order $\leq$ to $\leq'$ via the following rule: for any subsequence $b,c,a$ in $S$ with $b,c\neq a$, we require that $b\geq c$.
    If such extension is not possible, the reduction fails.
    \item Remove all $a$'s from $S$ to get $S'$, and also restrict $\leq'$ to $\Gamma\backslash \{a\}$.
    The pair $(S',\leq')$ is the output of the reduction.
\end{enumerate}
The following lemma shows that if we have decided that $(S,\leq)$ is $a$-decided, then we may reduce the foot-sortability of $(S,\leq)$ to a smaller problem using $a$-reduction.

\begin{lemma}\label{lemma:reduction}
    Let $S$ be a string representing a sock ordering, and $\leq$ be a partial order on the alphabet $\Gamma$.
    Suppose that $a\in \Gamma$ is such that $(S,\leq)$ is $a$-decided.
    If the $a$-reduction of $(S,\leq)$ fails, then $(S,\leq)$ is not foot-sortable.
    Otherwise, if $(S',\leq')$ is the output of the $a$-reduction, then $(S,\leq)$ is foot-sortable if and only if $(S',\leq')$ is foot-sortable.
\end{lemma}
\begin{proof}
    If $(S,\leq)$ is foot-sortable, then there exists a total order $\leq_{\ell}$ extending $\leq$ whose minimum element is $a$ such that $(S,\leq_{\ell})$ is foot-sortable.
    By \cref{theorem:criterion}, for any subsequence $b,c,a$ in $S$ with $b,c\neq a$, we must have $b\geq_{\ell} c$.
    Therefore the $a$-reduction must succeed, and $\leq_{\ell}$ extends $\leq'$.
    As $S'$ is a subsequence of $S$, we see that $(S',\leq_{\ell})$ is also foot-sortable by \cref{theorem:criterion}, and so $(S',\leq')$ is foot-sortable too.

    Suppose now that the $a$-reduction succeeds and $(S',\leq')$ is foot-sortable.
    Let $\leq_{\ell}'$ be a total order on the alphabet of $S'$ that extends $\leq'$ so that $(S',\leq_{\ell}')$ is foot-sortable.
    Let $\leq_{\ell}$ be the total order on the alphabet of $S$ that extends $\leq_{\ell}'$ by setting $a$ to be the minimum element.
    It is clear that $\leq_{\ell}$ extends $\leq$ as $a$ is a minimal element with respect to $\leq$.
    If $(S,\leq_{\ell})$ is not foot-sortable, then \cref{theorem:criterion} shows that there exists a subsequence $y,z,x$ of $S$ with $x<_{\ell}y<_{\ell}z$.
    As $(S',\leq_{\ell}')$ is foot-sortable, this shows that one of $x,y,z$ must be $a$.
    Since $a$ is the minimum element, we must have $a=x$.
    However as $\leq_{\ell}'$ extends $\leq'$, by the construction of $\leq'$ we must have $y\geq' z$, which is a contradiction.
    Therefore $(S,\leq_{\ell})$ is foot-sortable.
\end{proof}

Note that if $\leq'$ is obtained from taking a series of such reductions, then $\leq'$ is actually of a particular form.
We call a partial order $\leq$ on the alphabet of $S$ \emph{simple} if there is a prefix $S'$ of $S$ so that $a\geq b$ if and only if $a,b$ is a subsequence of $S'$.
In this case, we call the last letter of $S'$ the \emph{distinguished minimal letter}.
It is easy to see that if $(S,\leq)$ is such that $\leq$ is simple, and $(S',\leq')$ is the output of the $a$-reduction, then $\leq'$ is also simple.
As a consequence, we will focus on simple partial orders from now on.

\cref{lemma:reduction} allows us to focus on finding a letter $a$ so that $(S,\leq)$ is $a$-decided when $\leq$ is simple.
This will be done in the next subsection via a case analysis.

\subsection{Deciding the next minimum letter}
We first note that for any pair $(S,\leq)$, replacing any substring $aa$ with $a$ for any $a\in \Gamma$ does not affect its foot-sortability.
We can thus, without loss of generality, assume that there are no consecutive repeated letters.
We call the string $S$ \emph{reduced} if this is the case.
For a reduced string, call a letter \emph{lonely} if it appears only once in the string.
Call the letter \emph{unlonely} otherwise.

Now suppose that $(S,\leq)$ is such that $S$ is reduced and $\leq$ is simple.
It is clear that if all the letters are lonely, then $(S,\leq)$ is sortable.
Therefore we assume that there is some letter that is unlonely, and among all such letters, let $a$ be the one with the earliest second appearance.
Let $d$ be the distinguished minimal element.
As $S$ is reduced, we know that $d$ must appear before the second $a$.
Lastly, if it exists, let $b$ be the minimal lonely letter (with respect to $\leq$) with the earliest appearance in $S$.
In other words, $b$ has the earliest appearance among all lonely letters equal to $d$ or after $d$.
We will divide into cases based on the positions of $a,b,d$ in the string $S$.

\subsubsection{Case 1: $b$ exists and appears before the first $a$.}
This is the easiest case, and we can immediately show that in this case, $(S,\leq)$ is $b$-decided.
We start with two simple lemmas that will also be useful in the other cases.

\begin{lemma}\label{lemma:basic}
    Let $(S,\leq)$ be foot-sortable.
    For any three distinct letters $x,y,z$, if $x,y,x,z$ is a subsequence in $S$, then $z$ cannot be the minimum among $x,y,z$.
\end{lemma}
\begin{proof}
    Otherwise, suppose $\leq_{\ell}$ is a total order extending $\leq$ with $(S,\leq_{\ell})$ being foot-sortable.
    If $x>\leq_{\ell}y$ then the subsequence $y,x,z$ gives a contradiction.
    Otherwise the subsequence $x,y,z$ gives a contradiction.
\end{proof}

The second lemma says that in some scenario, we may move an element to the minimum in a total order while keeping foot-sortability.

\begin{lemma}\label{lemma:new-minimum}
    Let $(S,\leq_{\ell})$ be foot-sortable where $\leq_{\ell}$ is a total order.
    Let $m$ be the minimum element with respect to $\leq_{\ell}$, and let $m'$ be another letter so that there are no subsequences $m',z,m$ in $S$ with $z\neq m',m$.
    Let $\leq_{\ell}'$ be the total order obtained by making $m'$ minimum while keeping the remaining orders of $\leq_{\ell}$.
    Then $(S,\leq_{\ell}')$ is also foot-sortable.
\end{lemma}
\begin{proof}
    Suppose for the sake of contradiction that $(S,\leq_{\ell}')$ is not foot-sortable.
    Then by \cref{theorem:criterion}, there must be some subsequence $y,z,x$ of $S$ with $x<_{\ell}'y<_{\ell'}z$.
    Since $(S,\leq_{\ell})$ is foot-sortable, we see that $m'\in\{x,y,z\}$, and so $x=m'$ by the minimality of $m'$.
    By the assumption, we know that $y\neq m$.
    Moreover, $z\neq m$ as $m$ is the second least letter with respect to $\leq_{\ell}'$.
    Lastly, the assumption shows that there cannot be subsequences $m,y,z,m'$ or $y,m,z,m'$.
    Therefore there must be $y,z,m,m'$ or $y,z,m',m$, showing the existence of the subsequence $y,z,m$ in $S$ regardless.
    This is now a contradiction as $m<_{\ell}y<_{\ell}z$.
\end{proof}

Those two lemma combined imply that $(S,\leq)$ is $b$-decided in this case.

\begin{lemma}\label{lemma:case1}
    Let $S$ be a reduced string that represents a sock ordering, and let $\leq$ be a simple order on the alphabet of $S$.
    Let $a$ be the unlonely letter with the earliest second appeareance.
    If the earliest minimal lonely letter $b$ in $S$ appears before the second $a$, then $(S,\leq)$ is $b$-decided.
\end{lemma}
\begin{proof}
    Suppose that $(S,\leq)$ is foot-sortable and $\leq_{\ell}$ is a total order that certifies this.
    Modify $\leq_{\ell}$ to get $\leq_{\ell}'$ by making $b$ the minimum element.
    We will show that $(S,\leq_{\ell}')$ is also foot-sortable.

    Let $m$ be the minimum element with respect to $\leq_{\ell}$.
    Then the last appearance of $m$ must be at or after $b$: this is clear by the definition of $b$ if $m$ is lonely, and by the definition of $a$ if $m$ is unlonely.
    Note that if $m$ is unlonely but the first $m$ appears before $b$, then by the definition of $a$ and the reducedness of $S$ there exists some $z\neq a,m$ so that there is a subsequence $a,z,a,m$.
    This is a contradiction with \cref{lemma:basic}.
    Then $(S,\leq_{\ell}')$ is foot-sortable by \cref{lemma:new-minimum}.
    
\end{proof}

\subsubsection{Case 2: $b$ does not exist or appears after the second $a$.}
This case is slightly more complicated.
In most cases, we will be able to show that $(S,\leq)$ is $a$-decided.
However, there is one exception that requires some more care.

\begin{lemma}\label{lemma:case2}
    Let $S$ be a reduced string that represents a sock ordering, and let $\leq$ be a simple order on the alphabet of $S$.
    Let $a$ be the unlonely letter with the earliest second appeareance.
    If there are no minimal lonely letters before the second $a$ in $S$, then $(S,\leq)$ is $a$-decided unless there is a single letter $z$ between the first two $a$'s.
    In the exceptional case, $(S,\leq)$ is $\{a,z\}$-decided.
\end{lemma}
\begin{proof}
    If $(S,\leq)$ is foot-sortable, let $\leq_{\ell}$ be the total order extending $\leq$ certifying the foot-sortability.
    Let $m$ be the minimum element with respect to $\leq_{\ell}$.
    If $m=a$ then we are already done.
    Now if $m\neq a$, then by the assumption we must have that the last appearance of $m$ is after the second appearance of $a$.
    Therefore by the contrapositive of \cref{lemma:basic}, $m$ must be the only letter between the two $a$'s.
    This shows that we must be in the exceptional case and $m=z$, as desired.
\end{proof}

\cref{lemma:case2} reduces the second case to the exceptional case where there is a single letter $z$ between the first two $a$'s.
Deciding whether $(S,\leq)$ is $a$-decided or $z$-decided (or both) takes a bit more effort, and we do this by an exhaustive case analysis below.
We first note that if not both of $a,z$ are minimal with respect to $\leq$, then there is nothing to decide.
This happens when $z$ is the distinguished minimal letter, making $a$ not minimal and $(S,\leq)$ $z$-decided.
Therefore we just need to take care of the case where $a,z$ are both minimal with respect to $\leq$.

\begin{lemma}\label{lemma:case2-exceptional}
    Suppose that we are in the exceptional case of \cref{lemma:case2}.
    Also suppose that $a,z$ are both minimal with respect to $\leq$, i.e. $z$ is not the distinguished minimal letter $d$.
    Consider the first letter $x\neq a,z$ after the substring $aza$ in $S$.
    \begin{enumerate}
    \item If both $a,z$ appear after $x$, then $(S,\leq')$ is decided by whoever appears first.
    \item If some of $a,z$ do not appear after $x$, then $(S,\leq')$ is decided by whoever does not appear.
    \end{enumerate}
\end{lemma}
\begin{proof}
    Let $\leq_{\ell}$ be a total order extending $\leq$ such that $(S,\leq_{\ell})$ is foot-sortable.
    By \cref{lemma:case2}, we may assume that either $a$ or $z$ is the minimum element with respect to $\leq_{\ell}$.

    In the first case, we deal with the case where $a$ appears before $z$ after $x$.
    The other case can be dealt with in the same way.
    Since $a$ appears before $z$ after $x$, we get a subsequence $a,x,a,z$.
    By \cref{lemma:basic}, this shows that $z$ cannot be the minimum element of $\leq_{\ell}$, and so $a$ must be the minimum element of $\leq_{\ell}$, as desired.

    For the last case, we deal with the case where $z$ does not appear after $x$.
    Again, the other case can be dealt with in the same way.
    If $z$ is already the minimum element with respect to $\leq_{\ell}$, then there is nothing to prove.
    Therefore we assume that $a$ is the minimum.
    Now in the total order, move $z$ to the minimum to get another total order $\leq_{\ell}'$.
    Note that this still extends $\leq$ as $z$ is a minimal element with respect to $\leq$.
    Also $(S,\leq_{\ell}')$ must be foot-sortable by \cref{lemma:new-minimum}.
    This concludes the proof.
\end{proof}

\subsubsection{Case 3: $b$ exists an appears between the first two $a$'s.}
This case bears some similarity with the second case, though the actual argument would be a bit more complicate.
We start with showing that $(S,\leq)$ is $\{a,b\}$-decided in this case.

\begin{lemma}\label{lemma:case3}
    Let $S$ be a reduced string that represents a sock ordering, and let $\leq$ be a simple partial order on the alphabet of $S$.
    Let $a$ be the unlonely letter with the earliest second appearance.
    Suppose $b$ is the first minimal lonely letter, and it is between the first two $a$'s.
    Then $(S,\leq)$ is $\{a,b\}$-decided.
\end{lemma}
\begin{proof}
    Suppose that $(S,\leq)$ is foot-sortable.
    Let $\leq_{\ell}$ be a total order that extends $\leq$ such that $(S,\leq_{\ell})$ is foot-sortable.
    Let $m$ be the minimum with respect to $\leq_{\ell}$.
    If $m\in \{a,b\}$ then we are already done, so let us assume that $m\neq a,b$.
    This shows that the last appearance of $m$ must be before the second $a$ by \cref{lemma:basic}.
    By the definition of $a$, this shows that $m$ is lonely, and by the definition of $b$ we know that $m$ appears after $b$.
    Therefore, by \cref{lemma:new-minimum}, we may modify $\leq_{\ell}$ by making $b$ the minimum to get $\leq_{\ell}'$ so that $(S,\leq_{\ell}')$ is foot-sortable.
    This concludes the proof.
    
\end{proof}

The lemma shows that it suffices to consider the two choices we may have.
Therefore, to proceed, we would either take an $a$-reduction or a $b$-reduction and continue.
For our convenience, if we decide to take an $a$-reduction in Case 3, then we say that we take an \emph{A-step}.
Otherwise we say that we take a \emph{B-step}.

Note that if $a$ is not minimal, then we can only take a B-step.
We record this as the following corollary.

\begin{corollary}\label{corollary:case3}
    In the setting of \cref{lemma:case3}, if $a$ is not a minimal element, or equivalently, if $d$ comes after the first $a$, then $(S,\leq)$ is $b$-decided.
\end{corollary}

Therefore, from now on, we assume that $a,b$ are both minimal with respect to $\leq$.
There is yet another case where we can, without loss of generality, just take a B-step.

\begin{lemma}\label{lemma:B-step}
    In the setting of \cref{lemma:case3}, if both $a,b$ are minimal with respect to $\leq$ and $b$ appears right after the first $a$ in $S$, then $(S,\leq)$ is $b$-decided.
\end{lemma}
\begin{proof}
    By \cref{lemma:case3}, we just need to deal with the case where there is some total order $\leq_{\ell}$ extending $\leq$ so that $(S,\leq_{\ell})$ is foot-sortable and $a$ is minimum with respect to $\leq_{\ell}$.
    As $b$ is minimal with respect to $\leq$, we may modify $\leq_{\ell}$ to get $\leq_{\ell}'$ extending $\leq$ so that $b$ is minimum with respect to $\leq_{\ell}'$ while keeping the remaining orders.
    We are then done by \cref{lemma:new-minimum}, which shows that $(S,\leq_{\ell}')$ is foot-sortable.
\end{proof}

Next, we show that in most of the remaining cases, $(S,\leq)$ is $a$-decided.
There is some exceptional case that we will need to deal with, and this is similar to \cref{lemma:case2}.

\begin{lemma}\label{lemma:A-step}
    In the setting of \cref{lemma:case3}, if both $a,b$ are minimal with respect to $\leq$, and $b$ is not right after the first $a$ in $S$, then $(S,\leq)$ is $a$-decided unless there is exactly one unlonely letter $z$ between the first two $a$'s.
\end{lemma}
\begin{proof}
    Let $z$ be the letter right after $a$.
    Note that by the definition of $b$, we have that $z$ is unlonely.
    Therefore once we take a B-step, $a$ would no longer be minimal as we would have $a>z$ after the B-step.
    Note that $S$ will still be reduced after the reduction, and the role of $a$ will not change: it will still be the unlonely letter with the earliest second appearance.
    It is also clear that now we are either in Case 2 or Case 3.
    Note that if we are in Case 3 again, then we are forced to take a B-step by \cref{corollary:case3}.
    It is clear that if we keep doing B-steps until we end up in Case 2, then we are simply removing all the lonely letters between the first two $a$'s via the corresponding reductions.
    Let $(S',\leq')$ be the result after the B-steps until we get to Case 2.
    Note that once we are in Case 2, \cref{lemma:case2} says that $(S',\leq')$ is $a$-decided unless we are in the exceptional case.
    In the non-exceptional case, this immediately shows that $(S',\leq')$ is not foot-sortable as $a$ is not minimal (recall that $a>z$).
    This shows that taking a B-step only makes sense if $z$ is the only unlonely letter between the first two $a$'s.
    
\end{proof}

To finish up, we need to deal with the exceptional case in \cref{lemma:A-step}.
This will be done by a similar argument as \cref{lemma:case2-exceptional}.

\begin{lemma}\label{lemma:case3-exceptional}
    Suppose that we are in the exceptional case in \cref{lemma:A-step}.
    Let $x$ be the first letter after the second $a$ that is not $a$ or $z$.
    \begin{enumerate}
        \item If after the second $a$, there is some $z$ appearing before some other $a$, then $(S,\leq)$ is $b$-decided.
        \item If there is no $z$ after $x$, then $(S,\leq)$ is $b$-decided.
        \item If none of the above holds, then $(S,\leq)$ is $a$-decided.
    \end{enumerate}
\end{lemma}
\begin{proof}
    As we are in the exceptional case, the substring strictly between the first two $a$'s is $zb_1b_2\cdots b_k$ for some lonely $b_1=b,\ldots, b_k$.
    Recall that if we take a B-step, then we are forced to reduce by $b_1,b_2,\ldots, b_k$ and then $z$.
    Assume that $(S,\leq)$ is foot-sortable.
    Let $\leq_{\ell}$ be a total order that certifies the foot-sortability of $(S,\leq)$.
    By \cref{lemma:case3}, we may assume that the minimum element of $\leq_{\ell}$ is either $a$ or $b$.

    In the first case, there is a subsequence $a,z,b_1,a,z,a$ of $S$.
    In particular, there is a subsequence $z,b_1,z,a$, showing that $a$ cannot be the minimum element of $\leq_{\ell}$.

    The second case is slightly more complicated, though the argument still roughly follows the proof of \cref{lemma:new-minimum}.
    We suppose that $a$ is the minimum element of $\leq_{\ell}$.
    Consider a new total order $\leq_{\ell}'$ with $b_1<_{\ell}'\cdots <_{\ell}'b_k<_{\ell}'z<_{\ell}'x$ for any $x\not\in \{z,b_1,\ldots, b_k\}$, and also $x<_{\ell}'y$ if and only if $x<_{\ell}y$ for any $x,y\not\in\{z,b_1,\ldots, b_k\}$.
    It is clear that $\leq_{\ell}'$ still extends $\leq$, and we need to show that $(S,\leq_{\ell}')$ is foot-sortable.
    Otherwise, by \cref{theorem:criterion}, there is a subsequence $q,r,p$ of $S$ with $p<_{\ell}'q<_{\ell}'r$.
    As $(S,\leq_{\ell})$ is foot-sortable, we know that $p\in \{z,b_1,\ldots,b_k\}$.
    It is clear that $q,r\not\in \{z,b_1,\ldots, b_k\}$ if $p=z$.
    If $p=b_i$, then $q\not\in \{b_1,\ldots, b_k\}$ and so $r\not\in \{z,b_1,\ldots, b_k\}$.
    As $q$ appears before $r$, this forces $q\not\in \{z,b_1,\ldots, b_k\}$.
    This ensures the existence of the subsequence $q,r,a$.
    By the minimality of $a$ with respect to $\leq$, we know that $r\neq a$.
    We also have $q\neq a$ as otherwise $r$ will be forced to be in $\{z,b_1,\ldots, b_k\}$.
    This is now a contradiction as $a<_{\ell} q<_{\ell} r$ by the minimality of $a$.
    Thus $(S,\leq_{\ell}')$ is foot-sortable.

    For the last case, assume that $\leq_{\ell}$ has $b_1$ as the minimum element.
    Then we may, without loss of generality, assume that $b_1<_{\ell}\cdots <_{\ell} b_k < z$ are the first few elements in the total order.
    Let $\leq_{\ell}'$ be such that $a<_{\ell}' b_k<_{\ell}'\cdots <_{\ell}' b_1 < z$ are the first few elements, and the rest of the order is unchanged.
    If is clear that $\leq_{\ell}'$ still extends $\leq$.
    If $(S,\leq_{\ell}')$ is not foot-sortable, then again by \cref{theorem:criterion} there is some subsequence $q,r,p$ with $p<_{\ell}'q<_{\ell}'r$ and $p\in \{a,b_1,\ldots, b_k\}.$
    We first note that if some $a$ appears after $x$, then the assumption of this case shows that this $a$ must be before some $z$, implying that there is a subsequence $a,x,a,z$.
    This is a contradiction with \cref{lemma:basic}, and so there are no $a$'s after $x$.
    Therefore either $q\in \{a,z,b_1,\ldots, b_k\}$ or $q,r$ both appear before the first $a$.
    The latter is impossible as otherwise the subsequence $q,r,b_1$ already shows that $(S,\leq_{\ell})$ is not foot-sortable.
    However, if $q\in \{a,z,b_1,\ldots,b_k\}$, then we also have $r\in \{a,z,b_1,\ldots, b_k\}$.
    If we delete all other letters in $S$, we get the subsequence $S'=azb_1b\cdots b_kazzz\cdots z$.
    It is easy to verify that $(S',\leq_{\ell}')$ is foot-sortable, which contradicts the existence of $p,q,r$.
    Therefore $(S,\leq_{\ell}')$ has to be foot-sortable, as desired.
\end{proof}

\subsection{Algorithm and analysis}\label{subsection:algorithm}
With all the ideas above, we can already implement an algorithm for deciding foot-sortability easily.
Here, instead of implementing the algorithm na\"ively, we show how one can carefully design the algorithm so that the run time is $O(N\log N)$, where $N$ is the length of the sock ordering.
For readers that are only interested in a polynomial-time algorithm instead of the implementation detail for the faster algorithm that runs in $O(N\log N)$ time, the key ideas have already been all stated above, and the main procedure \cref{procedure:main} shows more in details how those ideas can be combined to give an algorithm. 

To avoid distraction, we will only focus on deciding whether a given sock ordering is foot-sortable or not, but it is easy to modify the algorithm slightly so that if the sock ordering is foot-sortable, the algorithm also outputs a certificate $\leq_{\ell}$.

\subsubsection{Input}
For simplicity, we will assume that the sock ordering is given as a sequence $s[0],\ldots, s[N-1]$ where each $s[i]$ is a non-negative integer in $[0,N)$ that represents the color of the $i$-th sock.

\subsubsection{Data structures}
We maintain several data structures to accelerate the procedures that we will introduce later.
First, we maintain a doubly linked list $L$ that represents the current state of the sock ordering.
In addition, we keep track of the position of the distinguished minimal element $DistMinPos$.
We also maintain an array of dictionaries implemented by self-balancing binary search trees (BST) ${Pos}[0],\ldots, {Pos}[N-1]$ where the key represents a position, and ${Pos}[j][i]$ is the pointer to the $i$-th sock in the list $L$ if $s_i=j$.
Lastly, we maintain a dictionary ${AllPos}$ that combines all the dictionaries above, a dictionary ${LonelyPos}$ that only combine the ${Pos}[j]$'s for lonely $j$'s, and a dictionary ${SecondPos}$ that only records the second appearances of the unlonely colors.

\subsubsection{Procedures}
We first explain how we preprocess the input to initialize the various data structures listed above.
Along the way, we can also make the string reduced, and we will also maintain the reducedness of the string starting from there.

\begin{algorithm}[H]
\floatname{algorithm}{Procedure}
\caption{\textsc{Preprocess}()}

\begin{algorithmic}[1]
    \STATE $DistMinPos \leftarrow 0$
    \STATE $LastColor \leftarrow \emptyset$
    \FOR{$i$ from $0$ to $N-1$}
        \IF{$s[i]\neq LastColor$}
            \STATE Append node $n=(s[i],i)$ to $L$
            \STATE $Pos[s[i]][i]\leftarrow {}^*n$, $Allpos[i]\leftarrow {}^*n$
            \STATE $LastColor\leftarrow s[i]$
            \IF{$\textsc{size}(Pos[s[i]])=1$}
                \STATE $LonelyPos[i]\leftarrow {}^*n$
            \ELSIF{$\textsc{size}(Pos[s[i]])=2$}
                \STATE $LonelyPos.\textsc{delete}({}^*Pos[s[i]].\textsc{front}.\textsc{second})$
                \COMMENT{Delete the first appearance of $s[i]$}
                \STATE $SecondPos[i]\leftarrow {}^*n$
            \ENDIF
        \ENDIF
    \ENDFOR
\end{algorithmic}
\end{algorithm}

Recall that throughout the algorithm, we will do $a$-reductions for various colors $a$.
Therefore it is useful to have a subroutine that deletes a number from the sequence while keeping the string reduced, and a procedure that uses this subroutine to do the $a$-reduction for some $a$.
In the following, the pointer $p$ points to the node that we wish to delete from $L$.

\begin{algorithm}[H]
\floatname{algorithm}{Procedure}
\caption{\textsc{Delete}(Pointer $p$)}

\begin{algorithmic}[1]
    \STATE $(Color,CurrentPos)\leftarrow {}^*p$
    \IF{$\textsc{size}(Pos[Color])=1$}
        \STATE $LonelyPos.\textsc{delete}(CurrentPos)$
    \ELSIF{$CurrentPos$ is the first or second key in $Pos[Color]$}
        \STATE Delete the second position of $Pos[Color]$ from $SecondPos$
        \IF{$\textsc{size}(Pos[Color])\geq 3$}
            \STATE Insert the third position of $Pos[Color]$ to $SecondPos$
        \ELSIF{$\textsc{size}(Pos[Color])= 2$}
            \STATE Insert the other position of $Pos[Color]$ to $LonelyPos$
        \ENDIF
    \ENDIF
    \STATE $Pos[Color].\textsc{delete}(CurrentPos)$, $AllPos.\textsc{delete}(CurrentPos)$
    \STATE $Prevp \leftarrow $the pointer before $p$ in $L$, $Nextp\leftarrow $ the pointer after $p$ in $L$
    \STATE $L.\textsc{delete}(p)$
    \IF[if deleting $p$ makes the string non-reduced]{${}^*Prevp.\textsc{first} = {}^*Nextp.\textsc{first}$}
        \STATE $\textsc{Delete}(Nextp)$
    \ENDIF
\end{algorithmic}
\end{algorithm}

With this subroutine, it is easy to do an $a$-reduction for some $a$.

\begin{algorithm}[H]
\floatname{algorithm}{Procedure}
\caption{\textsc{Reduce}(Color $a$)}

\begin{algorithmic}[1]
    \STATE $DistMinPos\leftarrow$ the last position of $a$
    \WHILE{$Pos[a]$ is not empty}
        \STATE $p = Pos[a]$.\textsc{front}.\textsc{value}
        \STATE \textsc{Delete}(p)
    \ENDWHILE
\end{algorithmic}
\end{algorithm}

We now explain how to put the procedures and the results we have proved so far together to decide the foot-sortability of the input.
We will first focus on how to deal with the non-exceptional cases.
The detail for the exceptional cases will come later.

\begin{algorithm}[H]
\floatname{algorithm}{Procedure}
\caption{\textsc{Main}()}\label{procedure:main}

\begin{algorithmic}[1]
    \STATE \textsc{Preprocess()}
    \WHILE{$SecondPos$ is not empty}
        \STATE $(aSecPos, aSecPointer)\leftarrow SecondPos.\textsc{front}$
        \STATE $a\leftarrow {}^*aSecPointer.\textsc{first}$\COMMENT{$a$ has the earliest second appearance}
        \STATE $(aFirstPos, aFirstPointer)\leftarrow Pos[a].\textsc{front}$
        \STATE $(bPos,bPointer)\leftarrow\textup{first node in }LonelyPos\textup{ with key at least }DistMinPos$
        \STATE $b\leftarrow {}^*bPointer.\textsc{first}$\COMMENT{$b$ is the earliest minimal lonely color}
        \STATE $AllNum\leftarrow$ number of keys of $AllPos$ in the interval $(aFirstPos, aSecondPos)$
        \STATE $LonelyNum\leftarrow$ number of keys of $LonelyPos$ in the interval $(aFirstPos, aSecondPos)$
        \IF[if we are in Case 1]{$b$ exists \textbf{and} $bPos< aFirstPos$}
            \STATE $\textsc{Reduce}(b)$
        \ELSIF[if we are in Case 2]{$b$ does not exist \textbf{or} $bPos > aSecPos$}
            \IF[non-exceptional case in \cref{lemma:case2}]{$AllNum \neq 1$}
                \IF{$DistMinPos > aFirstPos$}
                    \STATE \textbf{return} \textbf{false}
                \ENDIF
                \STATE $\textsc{Reduce}(a)$
            \ELSE[exceptional case]
                \STATE $\textsc{Case2Exception}(a)$
                \COMMENT{this will be dealt with later}
            \ENDIF
        \ELSE[we must be in Case 3 here]
            \IF[\cref{corollary:case3}]{$DistMinPos > aFirstPos$ }
                \STATE $\textsc{Reduce}(b)$
            \ELSIF[\cref{lemma:B-step}]{ $bPointer$ is immediately after $aFirstPointer$ in $L$}
                \STATE $\textsc{Reduce}(b)$
            
            \ELSIF[non-exceptional case in \cref{lemma:A-step}]{$AllNum \neq LonelyNum+1$}
                \STATE $\textsc{Reduce}(a)$
            \ELSE[exceptional case]
                \STATE $\textsc{Case3Exception}(a)$
                \COMMENT{this will be dealt with later}
            \ENDIF
        \ENDIF
        \algstore{myalg}
\end{algorithmic}
\end{algorithm}

\begin{algorithm}[H]
\begin{algorithmic}[1]
\algrestore{myalg}
        \IF{$SecondPos.\textsc{front}.\textsc{key} \leq DistMinPos$}
            \STATE \textbf{return} \textbf{false}
            \COMMENT{in this case, the new simple partial order is not a partial order}
        \ENDIF
    \ENDWHILE
    \STATE \textbf{return} \textbf{true}
\end{algorithmic}
\end{algorithm}

The exceptional case in Case 2 is a little bit tricky to deal with.
The main difficulty lies in finding the first $x\neq a,z$ after the second $a$, as required in \cref{lemma:case2-exceptional}.
However, as we will do an $a$-reduction or a $z$-reduction at the end, if we see any $a$ or $z$ right after the second $a$, they are redundant anyways and can be removed immediately.
This would limit the run time when amortized.

\begin{algorithm}[H]\
\floatname{algorithm}{Procedure}
\caption{\textsc{Case2Exception}(Color $a$)}\label{procedure:Case2Exception}

\begin{algorithmic}[1]
    \STATE $aFirstPos\leftarrow$ first key in $Pos[a]$, $aSecPos\leftarrow$ second key in $Pos[a]$
    \STATE $z\leftarrow $ the color corresponds to the node of $Pos$ in $(aFirstPos,aSecPos)$
    \IF[$a$ is not minimal]{$ DistMinPos > aFirstPos $}
        \STATE $\textsc{Reduce}(z)$
    \ELSE
        \STATE ${}^*xPointer\leftarrow $the pointer to the node after the second $a$ in $L$
        \WHILE[false if $xPointer$ does not exist]{${}^*xPointer.\textsc{first}\in \{a,z\}$}
            \STATE $\textsc{Delete}(xPointer)$
            \STATE $xPointer\leftarrow $the pointer to the node after the second $a$ in $L$
        \ENDWHILE
        \STATE $xPos\leftarrow xPointer.\textsc{second}$ \COMMENT{the position of first $x\neq a,z$ after $aSecPos$}
        \STATE $aPos\leftarrow \textup{first key in }Pos[a]\textup{ after }xPos$, $zPos\leftarrow \textup{first key in }Pos[z]\textup{ after }xPos$
        \IF[(2) in \cref{lemma:case2-exceptional}]{$aPos$ does not exist}
            \STATE $\textsc{reduce}(a)$
        \ELSIF[(2) in \cref{lemma:case2-exceptional}]{$zPos$ does not exist}
            \STATE $\textsc{reduce}(z)$
        \ELSIF[(1) in \cref{lemma:case2-exceptional}]{$aPos<zPos$}
            \STATE$\textsc{reduce}(a)$
        \ELSE[(1) in \cref{lemma:case2-exceptional}]
            \STATE $\textsc{reduce}(z)$
        \ENDIF
    \ENDIF
\end{algorithmic}
\end{algorithm}

The exceptional case in Case 3 is easier to dealt with algorithmically.
We simply extract the necessary information needed for \cref{lemma:case3-exceptional} and branch accordingly.

\begin{algorithm}[H]\
\floatname{algorithm}{Procedure}
\caption{\textsc{Case3Exception}(Color $a$)}\label{procedure:Case3Exception}

\begin{algorithmic}[1]
    \STATE $aFirstPos\leftarrow$ first key in $Pos[a]$, $aSecPos\leftarrow$ second key in $Pos[a]$
    \STATE $z\leftarrow $ the color corresponds to the unique node of $AllPos$ with key right after $aFirstPos$
    \STATE $aLastPos\leftarrow \textup{Last key in }Pos[a]\textup{ after }aSecPos$, $zPos\leftarrow \textup{first key in }Pos[z]\textup{ after }aSecPos$
        \algstore{myalg}
\end{algorithmic}
\end{algorithm}

\begin{algorithm}[H]
\begin{algorithmic}[1]
\algrestore{myalg}
    \IF[(1) in \cref{lemma:case3-exceptional}]{$aLastPos$ exists and $zPos < aLastPos$}
        \STATE $\textsc{Reduce}(b)$
    \ELSE
        \STATE $xPointer\leftarrow $the pointer to the node after the second $a$ in $L$
        \IF[false if $xPointer$ does not exist]{${}^*xPointer.\textsc{first}=z$}
            \STATE $xPointer\leftarrow $the pointer to the node after $xPointer$ in $L$
        \ENDIF
        \STATE $xPos\leftarrow {}^*xPointer.\textsc{second}$ \COMMENT{the position of first $x\neq a,z$ after $aSecPos$}
        
        \IF[(2) in \cref{lemma:case3-exceptional}]{There is no key in $Pos[z]$ above $xPos$}
            \STATE $\textsc{reduce}(b)$
        \ELSE[(3) in \cref{lemma:case3-exceptional}]
            \STATE $\textsc{reduce}(a)$
        \ENDIF
    \ENDIF
\end{algorithmic}
\end{algorithm}

\subsubsection{Correctness}
In this subsection, we analyze the algorithm and show that it is correct.
It is clear, from the pseudocode, that after preprocessing, the desired properties of the data structures hold.
One can also verify that the $\textsc{Delete}$ procedure maintains all the properties and the reducedness of $L$, and when doing a reduction, the variable $DistMinPos$ is set to the right position as long as the first position of $a$ is at least $DistMinPos$, which should always be the case.
It is also clear that the $\textsc{Reduce}$ procedure removes all the nodes with the prescribed color.
It remains to verify the validity of $\textsc{Main}$.

We denote by $S$ the current reduced string and $\leq$ the simple partial order determined by $DistMinPos$.
We would insist that whenever we start the while loop in Line $2$, $\leq$ has to be an actual partial order.
First, we note that $SecondPos$ is empty if and only if there are only lonely letters, in which case $(S,\leq)$ is always foot-sortable.
Therefore we can return true immediately if we are outside the while loop.
Otherwise, we are in the while loop, and we will find a letter $c$ so that $(S,\leq)$ is $c$-decided, make sure that $c$ is still minimal with respect to $\leq$, and do a $c$-reduction.
Whether it fails or not is verified at the end, at Line 32.
Note that the only way the reduction fails is that the new $\leq$ fails to be a partial order, which is only possible if some letter has its second appearance at or before the distinguished minimal element as $S$ is reduced.
By the definition of $(S,\leq)$ being $c$-decided, this will ensure that the foot-sortability of $(S,\leq)$ has not been changed.

It is easy to verify that the variables from Line 4 to Line 7 are what they are supposed to be.
The variable $AllNum$ counts the number of letters between the first two $a$'s, and $LonelyNum$ counts the number of lonely letters between the first two $a$'s.
Line 10 brings us to Case 1, and so $(S,\leq)$ is $b$-decided by \cref{lemma:case1}.
Note that $b$ is minimal with respect to $\leq$ by definition, so we do not need to verify that here.

Line 12 leads to Case 2.
If Line 13 holds, then \cref{lemma:case2} shows that $(S,\leq)$ is $a$-decided.
We need to test if $a$ is still minimal with respect to $\leq$.
If not, $(S,\leq)$ cannot be foot-sortable so we can immediately return false.
Otherwise, we do the $a$-reduction.
The case where Line 13 does not hold is the exceptional case in \cref{lemma:case2}.
The procedure $\textsc{Case2Exception}$ is invoked in this case, and we move our focus there.
Here, we know that $z$ must be minimal as otherwise the substring $aza$ forces $a>z>a$, which contradicts the assumption that $\leq$ is a well-defined simple partial order.
If $a$ is not minimial, then we can immediately do a $z$-reduction as in Line 4.
Otherwise, we would like to find the first $x\neq a,z$ after the second $a$.
We delete the $a$'s and $z$'s along the way in Line 8, which is fine as they would have been deleted anyways by an $a$-reduction or $z$-reduction.
The validity of the rest is guaranteed by \cref{lemma:case2-exceptional}.

Lastly, if we reach Line 21 in the main procedure, then we must be in the remaining case, which is Case 3.
\cref{lemma:case3} shows that $(S,\leq)$ is $\{a,b\}$-decided, and so if $a$ is not minimal then we can immediately take a $b$-reduction as in Line 23.
\cref{lemma:B-step} also ensures that we can take a $b$-reduction as long as the $b$ appears right after the first $a$.
Now if $AllNum\neq LonelyNum+1$, that means that the number of unlonely letters between the first two $a$'s is not $1$.
Therefore \cref{lemma:A-step} tells us that we can take an $a$-reduction at Line 27, where we know that $a$ has to be minimal.
The remaining case is the exceptional case, and we move on to verify the validity of the procedure $\textsc{Case3Exception}$.
Note that in this case, the letter $z$ right after the first $a$ must be unlonely, and thus the unique unlonely letter between the first two $a$'s.
Moreover, there must be a second $z$ after the second $a$ by the definition of $a$ and the unloneliness of $z$.
Hence, $zPos$ must always exist.
\cref{lemma:case3-exceptional} shows that if the last $a$ after the second, if it exists, comes after $zPos$, then $(S,\leq)$ is $b$-decided.
Therefore Line 5 is valid.
If we get into Line 6, then the substring starting from the second $a$ in $S$ must be of the form $a$, $ax\cdots$ or $azx\cdots$ for some $x\neq a,z$.
Therefore $xPointer$ points to the node after $a$ that is not $a$ or $z$.
The validity of the rest of the procedure follows from \cref{lemma:case3-exceptional}.

\subsubsection{Runtime Analysis}
Lastly, we provide a brief runtime analysis of the algorithm. 
By using some standard self-balancing BST and some standard techniques, we can make sure that for the dictionaries, the following operations all take $O(\log N)$ time: insertion, deletion, binary-search with keys, and counting keys in an interval.
Moreover, appending and deleting any node from $L$ takes $O(1)$ time.
It is then clear that $\textsc{Preprocessing}$ takes $O(N\log N)$ time, and running $\textsc{Delete}$ once takes $O(\log N)$ time, aside from the recursive calls.
Note that $\textsc{Delete}$ can be run for at most $N$ times, so the total time spent in $\textsc{Delete}$ is also $O(N\log N)$.
From now on, we can safely disregard the time cost of any call of $\textsc{Delete}$.
Using the same argument, we see that the total time spent in $\textsc{Reduce}$ (aside from $\textsc{Delete}$) is $O(N)$, and we also disregard the cost of $\textsc{Reduce}$ from now on.
Lastly, for the main procedure, most operations in the while loop are basic operations or standard operations for the dictionaries, and they take $O(\log N)$ time.
The only exception is the while loop in $\textsc{Case2Exception}$.
However, the cost of that while loop can be bounded, up to a constant, by the cost of $\textsc{Delete}$.
Therefore the amortized run time of a single while loop of the main procedure is $O(\log N)$.
This shows that the total run time is $O(N\log N)$, as desired.

With all the results we have shown, we have proved \cref{theorem:algorithm}.
\section{Minimal 2-bounded non-foot-sortable sock orderings}\label{section:min-2-bdd}
In this section, we list all the minimal non-foot-sortable sock orderings that are 2-bounded.
Intuitively, this is done by running the algorithm and see what are the relevant letters that fail the algorithm.
In some sense, we are looking for a potentially smaller subsequence that serves as a certificate that certifies the unsortability.
This idea can in principle be done for general sock orderings.
However, an exhaustive discussion would be too lengthy and complicated, and there are already several important ideas that are needed to deal with the 2-bounded sock orderings.
Therefore, we will only focus on 2-bounded sock orderings in this section.

In the following two subsections, we will first list some minimal non-foot-sortable sock orderings and verify that they are indeed minimal non-foot-sortable.
We will then show that they are indeed all such sock orderings by examining the algorithm and enumerate all possible ways in which the algorithm outputs false.

For our convenience, we will call the minimal 2-bounded non-foot-sortable sock orderings \emph{critical sock orderings}.

\subsection{The list of minimal 2-bounded non-foot-sortable sock orderings}

We begin by listing several sporadic critical sock orderings.
We will split them into four types.
How this classification is done will be clear later on.

\begin{lemma}[Sporadic critical sock orderings]\label{lemma:sporadic}
    The following $14$ sock orderings are all critical.
    \begin{itemize}
        \item[\textup{Type I.}] $abcdbacd,\, abcdedabc$;
        \item[\textup{Type I'.}] $abcadbdc,\, abcbdadc,\, abcdbadc,\, abcdcadb,\, abcdceaeb,\, abcdedacb$;
        \item[\textup{Type II.}] $abcdbcad,\, abcdcbad,\, abcdedbac$;
        \item[\textup{Type III.}] $abcabdedc,\, abcdadedc,\, abcdcaefeb$.
    \end{itemize}
\end{lemma}
\begin{proof}
    To see that they are all non-foot-sortable, we can simply run the algorithm on them.
    Take the first sock ordering $abcdbacd$ for example, $b$ has the earliest second appearance, and there are no lonely letters, showing that we are in Case 2.
    It is also the non-exceptional case, so it is $b$-decided.
    After a $b$-reduction, we get $acdacd$ with $a>c>d$.
    Now $a$ has the earliest second appearance, and we are still in the non-exceptional case in Case 2.
    However, $a$ is not minimal, so the algorithm outputs false here.
    Similar arguments can be done for all other sock orderings listed here.

    To show that they are minimal, one can enumerate through all socks, delete them from the sock ordering, and run the algorithm to verify that it becomes foot-sortable.
    Alternatively, we will later show that non-foot-sortable sock orderings must contain some sock ordering listed in this subsection as a subsequence.
    Therefore, once the list is complete, we can verify easily that no proper subsequence of them appears in the list.
\end{proof}

Aside from the $14$ sporadic sock orderings that are critical, there are four infinite families of critical sock orderings.
We list them below and show that they are all critical.

\begin{lemma}[Four families of critical sock orderings]\label{lemma:infinite}
The following four families all consist of critical sock orderings.
\begin{itemize}
    \item[\textup{Type A.}]$xa_0ya_{n-1}xa_{n-2}a_{n-1}a_{n-3}a_{n-2}\cdots a_0a_1$ where $n\geq 2$;
    \item[\textup{Type B.}]$a_0xya_{n-1}xya_{n-2}a_{n-1}a_{n-3}a_{n-2}\cdots a_0a_1$ where $n\geq 3$;
    \item[\textup{Type B'.}]$a_0yxa_{n-1}xya_{n-2}a_{n-1}a_{n-3}a_{n-2}\cdots a_0a_1$ where $n\geq 3$;
    \item[\textup{Type C.}]$a_0xa_{n-1}yzyxa_{n-2}a_{n-1}a_{n-3}a_{n-2}\cdots a_0a_1$ where $n\geq 3$.
\end{itemize}
    
\end{lemma}
\begin{proof}
    We begin by showing that any type-A sock ordering is not foot-sortable.
    By \cref{lemma:case2}, the sock ordering $xa_0ya_{n-1}xa_{n-2}a_{n-1}a_{n-3}a_{n-2}\cdots a_0a_1$ is $x$-decided.
    After an $x$-reduction, we get $a_0ya_{n-1}a_{n-2}a_{n-1}a_{n-3}a_{n-2}\cdots a_0a_1$ with $a_0>y>a_{n-1}$.
    This is the exceptional case of \cref{lemma:case2}, and by \cref{lemma:case2-exceptional} we see that it is $a_{n-1}$-decided.
    After doing an $a_{n-1}$-reduction, we get the string $a_0ya_{n-2}a_{n-3}a_{n-2}a_{n-4}a_{n-3}\cdots a_0a_1$ with $a_0>y>a_{n-2}$.
    By a simple induction, we can show that eventually we reduce the problem to $a_0ya_1a_0a_1$ with $a_0>y>a_1$.
    By \cref{lemma:case2}, $(a_0ya_1a_0a_1,\leq)$ is $a_0$-decided, which shows that it is not foot-sortable.
    Therefore the original sock ordering is also not foot-sortable.

    To show that any type-A sock ordering is critical, it remains to show that it becomes foot-sortable upon deletion of any sock.
    Consider the following grouping: $(xa_0)(y)(a_{n-1} x)(a_{n-2}a_{n-1})\cdots (a_0a_1)$.
    For each group, we specify a total order that certifies the foot-sortability when one of the element in the group is deleted.
    With the total order, one can then verify the foot-sortability by \cref{theorem:criterion}.
    For the group $(xa_0)$, the total order $y<x<a_{n-1}<\cdots <a_0$ serves as a certificate.
    For the singleton $(y)$, the total order $x<a_{n-1}<\cdots < a_0$ does the job.
    For $(a_{n-1} x)$, we can take $x<y<a_{n-1}<\cdots <a_0$.
    Finally, for the group $(a_ia_{i+1})$, the total order $y<a_{n-1}<\cdots <a_{i+1}<x<a_i<\cdots < a_0$ certifies the foot-sortability.

    The proofs for the other three families follows the same vein.
    For any type-B sock ordering, by \cref{lemma:case2} it must be $x$-decided.
    After the $x$-reduction, the sock ordering becomes $a_0ya_{n-1}ya_{n-2}a_{n-1}\cdots a_0a_1$ with $a_0>y>a_{n-1}$.
    This is not foot-sortable even without the second $y$, so the original type-B sock ordering is also not foot-sortable.
    The same argument works for type-B' sock orderings.

    For both type-B and type-B' sock orderings, we can also group $(a_{n-2}a_{n-1})(a_{n-3}a_{n-2})\cdots (a_0a_1)$.
    Whenever a sock is removed from the group $(a_ia_{i+1})$, we can show the foot-sortability by the total order $x<a_{n-1}<\cdots <a_{i+1}<y<a_i<\cdots <a_0$.
    If the first $a_0$ is deleted, then $x<a_{n-1}<\cdots<a_0<y$ works.
    Lastly, $x<y<a_{n-1}<\cdots <a_0$ takes care of the case when the first $y$ or the second $x$ is deleted, and $y<x<a_{n-1}<\cdots <a_0$ takes care of the case when the second $y$, the first $x$ or the first $a_{n-1}$ is deleted.

    The argument for type-C sock orderings is slightly more complicated.
    First, it is $z$-decided by \cref{lemma:B-step}, and it is $y$-decided by \cref{corollary:case3} after $z$-reduction.
    Once we take $y$-reduction, the instance becomes $a_0xa_{n-1}xa_{n-2}a_{n-1}\dots a_0a_1$ with $a_0>x>a_{n-1}$, which is not foot-sortable as shown above (by relabeling $x$ as $y$).
    To show that it is critical, we once again consider the grouping $(a_0)(xa_{n-1})(yz)(yx)(a_{n-2}a_{n-1})\cdots (a_0a_1)$.
    When $(a_0)$ is omitted, we can take $z<y<a_{n-1}<\cdots <a_0<x$.
    If some element in $(xa_{n-1})$ is missing, the total order $z<y<x<a_{n-1}<\cdots <a_0$ certifies the foot-sortability.
    For the group $(yz)$, the total order $x<y<z<a_{n-1}<\cdots <a_0$ works.
    For the next group $(yx)$, we may take $x<z<y<a_{n-1}<\cdots <a_0$ to show foot-sortability.
    Lastly, if some sock in $(a_ia_{i+1})$ is deleted, then we can take $z<y<a_{n-1}<a_{n-2}<\cdots <a_{i+1}<x<a_i<\cdots <a_0$.
    This shows that any type-C sock ordering is critical.
\end{proof}

Note that this already shows that once we remove the alignment-free assumption imposed in \cite{DK22}, we get infinitely many subsequences to avoid for foot-sortable sock orderings, which answers Question 5.1 Defant and Kravitz ask in \cite{DK22}.
This is really different from the alignment-free case, where Defant and Kravitz show that the alignment-free foot-sortable two-uniform sock orderings can be described as the ones avoiding a finite list of subsequences \cite{DK22}.

\subsection{The list is exhaustive}
In this subsection, we investigate all the places the main procedure \cref{procedure:main} would be forced to output false.
This would allow us to enumerate all critical sock orderings.
We start with an intermediate result that justifies the classification in \cref{lemma:sporadic}.

\begin{lemma}\label{lemma:intermediate}
    Suppose that the algorithm \cref{procedure:main} runs on some 2-bounded input and outputs false, then at some stage the string contains one of the following patterns as a subsequence.
    \begin{enumerate}
        \item[\textup{Type I.}] $apqapq$ with $a>p$;
        \item[\textup{Type I'.}] $apqaqp$ with $a>p$;
        \item[\textup{Type II.}] $xapaxp$ with $x>a>p$;
        \item[\textup{Type III.}] $apaxyxp$ with $a>p$;
        \item[\textup{Type IV.}] $aqpaq$ with $a>q>p$ and $q$ lonely.
    \end{enumerate}
\end{lemma}
\begin{proof}
    We examine one by one where we make decisions in \cref{procedure:main} and see if any of them possibly reaches a contradiction and outputs false.
    The first decision is made in Case 1 (Line 11), which always gives no immediate contradiction.
    For Case 2, if $a$ is not minimal and we are not in the exceptional case, then we immediately output false as in Line 15.
    The fact that we are not in the exceptional case shows that there are two unlonely letters $p,q$ between the first two $a$'s, or there is a lonely letter $q$ and an unlonely letter $p$, and we are in Case 2 as $q>p$.
    We first deal with the case where there are two unlonely letters.
    Let $p$ be the first one and $q$ be the second one.
    Then we must have a subsequence of $apqapq$ or $apqaqp$, and we must have $a>p$ as $a$ is not minimal and we are in Case 2.
    This shows that a type-I or type-I' pattern exists.
    Now if there is a lonely letter $q$, an unlonely letter $p$ with $q>p$, then we get a subsequence $aqpap$ with $a>q>p$.
    This gives a type-IV pattern.
    This deals with the possible contradictions reached by Line 15.
    Another possible contradiction is that $a$ is minimal, but $\leq$ fails to be a partial order after $a$-reduction.
    However, as there are at most two $a$'s and $a$ has the earliest second appearance, this is not possible.

    In the exceptional case, we move our focus to \cref{procedure:Case2Exception}.
    The first kind of decision is that $a$ is not minimal, and we are forced to do a $z$-reduction as in Line 4.
    If after $z$-reduction and reducing the string, there are still some $x$ with its second appearance before the position of the last (i.e. the second) $z$, then we reach a contradiction.
    As the sock ordering is 2-bounded and $z$ is the only letter between the two $a$'s, this is only possible with the subsequence $xazaxz\, (x>a>z)$ or $azaxyxz\, a>z$.
    Relabeling $z$ as $p$, we get the type-II or type-III pattern.
    The remaining case is if $a$ is minimal.
    Note that since the sock ordering is two-bouonded, there cannot be anymore $a$ after the first $x\neq a,z$ after the second $a$, so we always reduce with $a$ at Line 14.
    Since we are reducing with $a$, as before, this never gives an immediate contradiction.

    For Case 3, similar to Case 1, we never get an immediate contradiction if we reduce with $b$.
    Also similar to Case 2, as the sock-ordering is 2-bounded, we never get an immediate contradiction if $a$ is minimal and we reduce with $a$.
    Therefore we never get an immediate contradiction in Case 3.
    This concludes the proof.
\end{proof}
\begin{remark}
    We use the fact crucially for several times that if the letter $a$ with the earliest second appearance is minimal and we reduce with $a$, then we never reach an immediate contradiction.
    This relies on the assumption that the sock ordering is 2-bounded.
    If we remove this assumption, then we also need to take care of other types of patterns, such as $apqapqxyxa$ with $a>p$.
    We will later see some other minimal non-foot-sortable sock ordering that is not 2-bounded.
    It seems that there would be many more case analyses needed once the 2-bounded assumption is lifted.
\end{remark}

It remains to see where the inequality conditions come from in different types of patterns.
To do so, we need the following lemma.

\begin{lemma}\label{lemma:no-contradiction}
    Run the algorithm \cref{procedure:main} on some input that is 2-bounded.
    Suppose that at some point, $a$ is the letter with the earliest second appearance, and $a$ is not larger than any unlonely letter.
    Then the algorithm will not output false as long as $a$ is still the letter with the earliest second appearance.
\end{lemma}
\begin{proof}
    Note that in all types of patterns listed in \cref{lemma:intermediate}, $a$ is always greater than some unlonely letter.
    Therefore when we start, we do not get an immediate contradiction.
    To verify whether this holds true in later iterations, we just need to examine whether this is true whenever we take a $b$-reduction, as taking an $a$-reduction or a $z$-reduction breaks the property of $a$.
    We will again refer to lines of \cref{procedure:main}.
    Note that the $b$-reductions in Case 1 (Line 11) and Line 25 in Case 3 maintain the property that $a$ is not greater than some unlonely letter.
    If we are forced to take a $b$-reduction because $a$ is not minimal as in Line 23, then the distinguished minimal element must be lonely, and thus it must be $b$.
    This shows that taking a $b$-reduction still does not make $a$ larger than any unlonely letter.
    Lastly, if we take a $b$-reduction in \cref{procedure:Case3Exception}, then as shown in \cref{lemma:case3-exceptional}, eventually we will get back to the exceptional case of Case 2 and take a $z$-reduction.
    As shown in \cref{lemma:intermediate}, this never gives an immediate contradiction.
    Therefore we would also not get a contradiction until $a$ is no longer the letter with the earliest second appearance.
\end{proof}

The contrapositive of \cref{lemma:no-contradiction} implies that in order for one of the patterns listed in \cref{lemma:intermediate} to exist, the letter $a$ must be larger than some unlonely letter as soon as $a$ becomes the letter with the earliest second appearance.
Note that in all cases, we can take the unlonely letter to be $p$ by taking $p$ to be the first unlonely letter after $a$.
Now we begin by dealing with the first three types in \cref{lemma:intermediate}, which would give rise to the list of sporadic critical sock orderings.
In the following, we will disregard all subsequences that have the type-A sequence $xa_0ya_1xa_0a_1$, as this would be included in the four infinite families anyways.

\begin{lemma}\label{lemma:sporadic-complete}
    Suppose we have a 2-bounded sock ordering that avoids the pattern $xa_0ya_1xa_0a_1$, and when the algorithm \cref{procedure:main} runs, at some point we find some pattern listed in \cref{lemma:intermediate} that is not of type IV.
    Then the original input contains one of the critical sock ordering listed in \cref{lemma:sporadic}.
\end{lemma}
\begin{proof}
    No matter which type of the pattern we find in \cref{lemma:intermediate}, by \cref{lemma:no-contradiction} and the discussion above, we know that we must have $a>p$ once $a$ becomes the letter with the earliest second appearance.
    Suppose that $a$ became the letter with the earliest second appearance after doing an $m$-reduction.
    Then the last $m$ must be before the second $a$ but after the first $p$.

    We first take type-I patterns as an example.
    If we find a type-I pattern at some point, then the last $m$ can be inserted into the pattern $apqapq\, (a>p)$ to get $apmqapq$ or $apqmapq$.
    Say it is $apmqapq$ for example.
    If $m$ is unlonely, then the first $m$ can be either be before $a$ (and we get $mapmqapq$), between $a$ and $p$ (and we get $ampmqapq$), or between $p$ and the first $m$.
    However, for the last case, there must be another letter $n$ between the two $m$'s, and we get $apmnmqapq$.
    Now if $m$ is lonely, then it must be the case that the letters before and after $m$ are the same, and we get $apnmnqapq$, which is $apmnmqapq$ by relabeling.  
    Using a similar argument, we see that if we encounter type-I pattern at some point, then the input must contain one of the following as a subsequence: $m\textcolor{red}{apmqapq}$, $\textcolor{red}{a}m\textcolor{red}{pmqapq}$, $\textcolor{red}{apm}nm\textcolor{red}{qapq}$, $\textcolor{red}{mapqma}p\textcolor{red}{q}$, $ampqmapq$, $\textcolor{red}{apmq}m\textcolor{red}{apq}$, $apqmnmapq$.
    However, the red subsequences are the same patterns as $xa_0ya_1xa_0a_1$.
    Therefore there are only two possibilities, and by relabeling we get the two critical sock orderings $abcdbacd$ and $abcdedabc$.

    The discussion for type I', type II and type III is also similar.
    For type I', using the same argument, the input must contain one of $mapmqaqp$, $ampmqaqp$, $apmnmqaqp$, $\textcolor{red}{mapqmaq}p$, $ampqmaqp$, $apmqmaqp$, $apqmnmaqp$.
    Discarding the one possibility that also contains $xa_0ya_1xa_0a_1$, we get six possibilities, and upon relabeling they are exactly the six type-I' critical sock orderings in \cref{lemma:sporadic}.
    Similarly, for type II, the only possibilities are $\textcolor{red}{mxapm}a\textcolor{red}{xp}$, $xmapmaxp$, $xampmaxp$ and $xapmnmaxp$.
    Discarding the first one, the remaining three are the three listed as type-II critical sock orderings in \cref{lemma:sporadic}.
    Lastly, for type III, the possibilities are $mapmaxyxp$, $ampmaxyxp$ and $apmnmaxyxp$, and they are exactly the three type-III critical sock orderings in \cref{lemma:sporadic} after relabeling.
    
\end{proof}
\begin{remark}
    It is not immediate that the subsequences obtained this way would not be foot-sortable (though those can be verified as stated in \cref{lemma:sporadic}).
    In fact, it is not necessarily the case for type-IV patterns, which is partly the reason that there are four infinite families in that case.
    The reason why the subsequences happen to be non-foot-sortable for type-I, type-I', type-II and type-III patterns is that the algorithm always reduces by $m$ before it reduces by $a,p,x$ or $y$ even without any condition.
    This will no longer be true for type-IV patterns, and we will see how this leads to the four infinite families.
\end{remark}

We next deal with the type-IV patterns.
We will assume for convenience that the input contains no sporadic critical sock orderings as subsequences, as they are already accounted for.
We first show the following result, and the argument there will also be useful later.

\begin{lemma}\label{lemma:infinite-intermediate}
    Suppose we have a 2-bounded sock ordering that avoids all the sporadic critical sock orderings listed in \cref{lemma:sporadic} as subsequences.
    If the algorithm \cref{procedure:main} runs on this sock ordering and encounters a type-IV pattern in \cref{lemma:intermediate}, then at some point in the algorithm, there is either a pattern $maqpmap$, or a pattern $aqmpmap$ with the $q$ being the first $q$, $q>m$ but $q\not>p$, and there is no $q$ after the first $p$.
\end{lemma}
\begin{proof}
    Consider the iteration of the while loop where the distinguished minimal element is first set to be something at or after the first $p$.
    Suppose that $a'$ is the letter with the earliest second appearance at that time, and $b'$ is the earliest minimal lonely letter if it exists.
    By the assumption, we know that whichever the letter we decide to reduce with, its last appearance must be after $p$.
    Moreover, $q$ must be lonely after the reduction of this iteration.

    Suppose that $q$ is already lonely before the reduction, and $q$ is greater than some unlonely letter $m$.
    As $q$ is not greater than $p$ in this iteration yet, we know that the first $m$ must be between $q$ and $p$.
    The second $m$ must be at or after the second $a'$.
    It should also be in front of the second $a$, as otherwisewe have found a type-I or type-I' pattern, which shows that some sporadic critical ordering exists.
    This will be useful later on.
    
    Suppose first that in this particular iteration, we are in Case 1.
    In this case, $a'\neq q$, and so $q$ must already be lonely.
    Since $q$ appears before $b'$, it must be the case that $q$ is greater than some unlonely letter $m$.
    Note that since we reduce with $b'$ in this step, $b'$ is after $p$.
    This shows that the second $m$ must be after $p$.
    Thus the only possibility is $aqmpmap$, with an extra condition $q>m$, as desired.

    If we are in Case 2, suppose first that we reduce with $a'$.
    Then $a'\neq q$, and $q$ must be lonely, showing that $q$ is already greater than some unlonely letter $m$.
    Note that the second $a'$ must be after $p$, and so as before, we would end up with $aqmpmap$ with $q>m$.
    Now if we are in the exceptional case, suppose that we reduce with $z'$.
    If $a'=q$, then we find $aqz'qpz'ap$ with $q>z'$, which contains $aqz'pz'ap$ with $q>z'$.
    Otherwise, we have $a'\textcolor{red}{z'}a'\textcolor{red}{aqpz'ap}$, $\textcolor{red}{aa'z'}a'q\textcolor{red}{pz'ap}$, or $\textcolor{red}{aq}a'\textcolor{red}{z'}a'\textcolor{red}{pz'ap}$, all with $a'>z'$.
    The first highlighted part is the same as $maqpmap$, and the second and third ones are $aqmpmap$ with $q>m$.

    Lastly, suppose that we are in Case 3.
    As the reduction taken at this step makes $q>p$, it must be the case that we take an A-step here and $a'$ is after $p$.
    If $q$ is before the first $a'$, then $q$ must be larger than some unlonely letter $m$, and the second $m$ must be after $a'$, which is after $p$.
    Therefore, as before, we get $aqmpmap$ with $q>m$.
    Now if $q$ is after the first $a'$, then we get $a'aqpa'ap$ or $aa'qpa'ap$.
    The first is $maqpmap$ after relabeling.
    For the second one, as we take an A-step instead of a B-step, there must be another unlonely letter $r$ between the first $a'$ and $q$ (as $q$ has to be minimal).
    If the second $r$ comes after the second $a$, then after taking the $a'$-reduction, we get a type-I or type-I' pattern (after discarding $q$), showing that the input already consists of some sporadic critical sock orderings.
    Therefore we must have $aa'rqpa'rap$.
    However, after taking an $a'$-reduction, we get $arprap$ with $a>r>p$, which is the type-II pattern.
    Thus the input also contains a sporadic critical sock ordering in this case.
    This concludes the case analysis.
    
\end{proof}

We can finally complete the classification of the infinite families of critical sock orderings with all the results we have shown.
The argument would look really similar to the one for \cref{lemma:infinite-intermediate}.
The difference is that we would have more options in several cases that lead to four different types of critical sock orderings.

\begin{lemma}\label{lemma:infinite-complete}
    Suppose we have a 2-bounded sock ordering that avoids all the sporadic critical sock orderings listed in \cref{lemma:sporadic} as subsequences.
    If the algorithm \cref{procedure:main} runs on this sock ordering and encounters a type-IV pattern in \cref{lemma:intermediate}, then the input contains one of the critical sock orderings listed in \cref{lemma:infinite} as subsequence.
\end{lemma}
\begin{proof}
    Run the algorithm, and find the moment where we can find a subsequence $a_0ya_{n-1}a_{n-2}a_{n-1}\cdots a_0a_1$ for some $n\geq 3$ with the earliest $a_{n-1}$ with the following conditions:
    \begin{itemize}
        \item All $y$'s appear before the first $a_{n-2}$.
        \item There is no $y$ before the $y$ in the subsequence.
        \item $y>a_{n-1}$.
    \end{itemize}
    If there are ties, we break ties with choosing the latest $a_0$.
    If there are still ties, choose an arbitrary one.
    By \cref{lemma:intermediate}, we may assume that $n\geq 3$ as $maqpmap$ is part of the list in \cref{lemma:infinite}, and $aqmpmap$ with $q>m$ is $a_0ya_2a_1a_2a_0a_1$ with $y>a_2$ after relabeling.
    We will follow the proof strategy of \cref{lemma:intermediate} from here.

    Consider the iteration of the while loop where the distinguished minimal element is first set after the first $a_{n-1}$.
    We still call the letter with the earliest second appearance $a'$, and the earliest minimal lonely letter $b'$ if it exists.

    We first deal with the case where $y$ is not lonely and we reduce with $a'$ or a lonely letter $b'$.
    In this case, the second $a'$ must be before the second $y$, which is before the first $a_{n-2}$.
    Moreover, the distinguished minimal element would always be before the second $a'$ after the reduction, showing that the second $a'$ must be after $a_{n-1}$.
    Therefore we must have $a_0a'ya_{n-1}a'ya_{n-2}a_{n-1}\cdots a_0a_1$, $a_0ya'a_{n-1}a'ya_{n-2}a_{n-1}\cdots a_0a_1$ or $a_0ya_{n-1}a'xa'ya_{n-2}a_{n-1}\cdots a_0a_1$.
    They are type-B, type-B' and type-C, respectively.

    If $y$ is already lonely, and $y$ is greater than some unlonely letter $m$, then as before, we know that the first $m$ must be between $y$ and $a_{n-1}$.
    We also know that the second $m$ must be after the second $a'$.
    Now if the second $m$ is after $a_0$, then as in \cref{lemma:infinite-intermediate} we have found a type-I or type-I' pattern, which guarantees the existence of some sporadic critical sock ordering by \cref{lemma:sporadic-complete}.
    If the second $m$ is before $a_0$ but after the first $a_{n-1}$, let $1\leq k\leq n-1$ be the least such that $m$ is after the first $a_k$.
    Then we get the subsequence $a_0yma_kma_{k-1}a_k\cdots a_0a_1$ with $y>m$, and we reach a contradiction as $m$ appears earlier than $a_{n-1}$.
    Therefore the second $m$ must be between $y$ and the first $a_{n-1}$ too. 
    This will again be useful later on.
    
    Suppose in this iteration, we are in Case 1.
    Then the second $a'$ must be after the first $a_{n-1}$.
    If $y$ is not lonely, then we get that the second $y$ is after the second $a'$ but before the first $a_{n-2}$.
    This gives $a_0ya_{n-1}a'xa'ya_{n-2}a_{n-1}\cdots a_0a_1$, which is of type $C$ by relabeling.
    Now if $y$ is lonely, then $y$ must be greater than some unlonely letter $m$, and by the discussion above, we see that the second $m$ is before the first $a_{n-1}$, which shows that the second $a'$ is also before the first $a_{n-1}$.
    This is a contradiction.

    Now suppose we are in Case 2.
    Suppose first that we reduce with $a'$.
    The case where $y$ is unlonely is already dealt with, so we can assume that $y$ is lonely.
    Then as the second $a'$ is after $y$, $y$ must be larger than some unlonely letter $m$.
    We will then reach a contradiction in the same way as in the paragraph above.
    The remaining case in Case 2 is the exceptional case where we reduce by $z'$.
    
    In the exceptional case, if $a'=y$, then we have $y>z'$ as we decide to reduce by $z'$.
    Though $y$ is not lonely in this case, using the same argument we can still show that the second $z'$ is after the first $a_{n-1}$ and before the first $a_{n-2}$, giving $a_0yz'ya_{n-1}za_{n-2}a_{n-1}\cdots a_0a_1$ with $y>z'$.
    If we omit the second $y$ and relabel $z'$ as $a_n$, we get $a_0ya_na_{n-1}a_n\cdots a_0a_1$ with $y>a_n$ and the second $y$ before $a_{n-1}$.
    This contradicts the choice of the subsequence we start with.
    
    If $a'\neq y$ instead, then we have $a'>z'$, and $a'z'a'$ either appear before $a_0$, between $a_0$ and $y$, or between $y$ and $a_{n-1}$.
    Moreover, as we cannot get an immediate contradiction upon reducing by $z'$, the last $z'$ has to be before the second $a_{n-1}$.
    Therefore there are six possibilities: 
    $a'\textcolor{red}{z'}a'\textcolor{red}{a_0ya_{n-1}z'a_{n-2}a_{n-1}\cdots a_0a_1}$, 
    $a'\textcolor{red}{z'}a'\textcolor{red}{a_0y}a_{n-1}\textcolor{red}{a_{n-2}z'}a_{n-1}\textcolor{red}{\cdots a_0a_1}$,
    $\textcolor{red}{a_0a'z'}a'y\textcolor{red}{a_{n-1}z'a_{n-2}a_{n-1}\cdots a_0a_1}$, 
    $\textcolor{red}{a_0a'z'}a'ya_{n-1}\textcolor{red}{a_{n-2}z'}a_{n-1}\textcolor{red}{\cdots a_0a_1}$, 
    $\textcolor{red}{a_0y}a'\textcolor{red}{z'}a'\textcolor{red}{a_{n-1}z'a_{n-2}a_{n-1}\cdots a_0a_1}$, 
    $\textcolor{red}{a_0y}a'\textcolor{red}{z'}a'a_{n-1}\textcolor{red}{a_{n-2}z'}a_{n-1}\textcolor{red}{\cdots a_0a_1}$.
    The highlighted parts of the first and second possibilities are already type-A critical sock orderings, and for the rest of the possibilities, the highlighted part is another sequence that we could have started with but with an earlier $a_{n-1}$, which is a contradiction.
    This concludes the discussion for Case 2.

    For Case 3, if we choose to do a B-step and reduce with a lonely letter $b'$, then $b'$ must be after $a_{n-1}$, and in particular, after $y$.
    The case where $y$ is unlonely is done, so we assume that $y$ is lonely.
    This means that $y>a'$, and we have shown that in this case, we must have that the second $a'$ must be between $y$ and the first $a_{n-1}$.
    This contradicts the fact that $b'$ is after $a_{n-1}$, as $b'$ must be before $a'$.
    Now if we choose to do an A-step, we may once again assume that $y$ is lonely.
    If $y$ is before the first $a'$, then it must be the case that $y$ is larger than some unlonely letter, and we can reach a contradiction in the same way as before.
    Therefore the first $a'$ is before $y$, and the second $a'$ is between the two $a_{n-1}$'s.
    If $a'$ is before $a_0$, then we get two possibilities: $a'a_0ya_{n-1}a'a_{n-2}a_{n-1}\cdots a_0a_1$ and  $a'a_0ya_{n-1}a_{n-2}a'a_{n-1}\cdots a_0a_1$.
    The first itself is type-A, and the second is also type-A once we remove the $a_{n-1}$'s.
    Lastly, if $a'$ is between $a_0$ and $y$, as we decide to take an A-step, there must be some unlonely letter $r$ between the first $a'$ and $y$.
    If the second $r$ is before $a_{n-1}$, then after doing an $a'$-reduction, we get $a_0>r>y$ and a subsequence $a_0ra_{n-1}ra_{n-1}a_0$, which is a type-II pattern and thus a contradiction.
    If the second $r$ is after the second $a_0$, then after the $a'$-reduction, we get $a_0>r$ and a subsequence $a_0ra_1a_0ra_1$ or $a_0ra_1a_0a_1r$, which is a type-I or type-I' pattern.
    Those are also contradictions.
    Let $k$ be the largest non-negative integer such that $r$ is before the second $a_k$.
    Then $0\leq k\leq n-2$, and after the $a'$-reduction, we get a subsequence of $rya_{n-1}a_{n-2}a_{n-1}\cdots a_ka_{k+1}ra_k$ with $y>a_{n-1}$.
    Note that this subsequence satisfies all the conditions that the original subsequence $a_0ya_{n-1}a_{n-2}a_{n-1}\cdots a_0a_1$ meet, they have the same $a_{n-1}$, and $r$ is later than $a_0$.
    Therefore this is a contradiction with the tie-breaking rule.
    This concludes the proof.
    
\end{proof}

We can now put everything together to prove \cref{theorem:all-minimal}.
\begin{proof}[Proof of \cref{theorem:all-minimal}]
    To show the theorem, we will show that any non-foot-sortable 2-bounded  sock ordering contains a pattern listed in \cref{lemma:sporadic} or \cref{lemma:infinite}.
    Otherwise, take a counterexample that does not contain any of the patterns.
    By \cref{lemma:intermediate} and the contrapositive of \cref{lemma:sporadic-complete}, when the algorithm \cref{procedure:main} runs on the counterexample, we must get a type-IV pattern at some point.
    \cref{lemma:infinite-complete} then leads to a contradiction.
\end{proof}
\section*{Acknowledgement}
I would like to thank Noah Kravitz for introducing this problem to me, and Colin Defant and Noah Kravitz for some discussions that are particularly helpful for Section 4.
\bibliographystyle{amsplain0}
\bibliography{ref_joints}

\providecommand{\bysame}{\leavevmode\hbox to3em{\hrulefill}\thinspace}
\providecommand{\MR}{\relax\ifhmode\unskip\space\fi MR }
\providecommand{\MRhref}[2]{%
  \href{http://www.ams.org/mathscinet-getitem?mr=#1}{#2}
}
\providecommand{\href}[2]{#2}
\begin{thebibliography}{1}

\bibitem{DK22}
Colin Defant and Noah Kravitz, \emph{Foot-sorting for socks}, 2022,
  arXiv:2211.02021.

\bibitem{Knuth}
Donald~E. Knuth, \emph{The art of computer programming. {V}ol. 3}, second ed.,
  Addison-Wesley, Reading, MA, 1998, Sorting and searching.

\end{thebibliography}
\end{document}